\documentclass[11pt]{article}

\usepackage[T1]{fontenc}
\usepackage{lmodern}

\usepackage{soul}
\usepackage{amsmath, amssymb, amsthm} 
\usepackage[margin=1in]{geometry}
\usepackage{color, graphicx}
\usepackage[breakable]{tcolorbox}

\usepackage[labelfont=sl,textfont=sl]{subcaption} % newer version of subfig
\usepackage[font={small,sl},labelsep=period]{caption} % format figure & table captions

\usepackage{algpseudocode} % algorithmicx
\usepackage{algorithm} % put algorithmic in a floating object
\algrenewcommand\algorithmicrequire{\textbf{Input:}}
\algrenewcommand\algorithmicensure{\textbf{Output:}}

\usepackage{etoolbox} 
\makeatletter
\patchcmd{\@maketitle}{\LARGE \@title}{\LARGE\bfseries\@title}{}{}

\renewcommand{\@seccntformat}[1]{\csname the#1\endcsname.\quad}
\makeatother

\definecolor{darkblue}{rgb}{0,0,.5}
\usepackage{hyperref}
\hypersetup{
	colorlinks=true,        % false: boxed links; true: colored links
	linkcolor=darkblue,     % color of internal links
	citecolor=darkblue,     % color of links to bibliography
	urlcolor=darkblue       % color of external links
}

\makeatletter

\newtheoremstyle{mytheorem}%                % Name
  {6pt} %% Space above
  {}    %% Space below
  {\itshape}%% Body font
  {}    %% Indent amount
  {\bfseries}%% Theorem head font
  {.} %% Punctuation after theorem head
  { } %% Space after theorem head, ' ', or \newline
  {\thmname{#1}\thmnumber{ #2}\thmnote{ (#3)}}  % Theorem head spec

\newtheoremstyle{mydefinition}%                % Name
  {6pt} %% Space above
  {}    %% Space below
  {}%% Body font
  {}    %% Indent amount
  {\bfseries}%% Theorem head font
  {.} %% Punctuation after theorem head
  { } %% Space after theorem head, ' ', or \newline
  {\thmname{#1}\thmnumber{ #2}\thmnote{ (#3)}}  % Theorem head spec

\theoremstyle{mytheorem}
\newtheorem{theorem}{Theorem}[section]
\newtheorem{lemma}[theorem]{Lemma}
\newtheorem{corollary}[theorem]{Corollary}
\newtheorem{proposition}[theorem]{Proposition}

\theoremstyle{mydefinition}
\newtheorem{assumption}[theorem]{Assumption}

\newtheorem{example}[theorem]{Example}
\newtheorem{remark}[theorem]{Remark}

\usepackage[shortlabels]{enumitem}
\setlist[enumerate]{nosep}

\newcounter{step}[theorem]
\newcommand\step[1]{%
	\refstepcounter{step}	
	\vskip 0.25\baselineskip
	\ifx\hfuzz#1\hfuzz
		\item[~~\(\triangleright\)~\textbf{Step~\arabic{step}.}]
	\else
		\item[~~\(\triangleright\)~\textbf{Step~\arabic{step}}] (\texttt{#1})\textbf{.}%
	\fi
}

\newcommand{\NN}{\ensuremath{\mathbb N}}
\newcommand{\RR}{\ensuremath{\mathbb R}}

\newcommand{\argmin}{\ensuremath{\operatorname*{argmin}}}

\newcommand{\dom}{\ensuremath{\operatorname{dom}}}

\newcommand{\epi}{\ensuremath{\operatorname{epi}}}

\newcommand{\prox}{\ensuremath{\operatorname{Prox}}}

\newcommand{\grad}{\nabla}
\newcommand{\gammainc}{\gamma_{\textnormal{inc}}}
\newcommand{\gammadec}{\gamma_{\textnormal{dec}}}

\definecolor{dgreen}{rgb}{0.0,0.7,0.0}

\begin{document}

\title{Projected proximal gradient trust-region algorithm\\ for nonsmooth optimization}

\author{
Minh N.~Dao\thanks{School of Science, RMIT University, Melbourne, VIC 3000, Australia. 
E-mail: \texttt{minh.dao@rmit.edu.au}.},
~
Hung M.~Phan\thanks{Department of Mathematics and Statistics, Kennedy College of Sciences, University of Massachusetts Lowell, Lowell, MA 01854, USA.
E-mail: \texttt{hung\char`_phan@uml.edu}.},
~and~
Lindon Roberts\thanks{School of Mathematics and Statistics, University of Sydney, Camperdown NSW 2006, Australia.\\ Email: \texttt{lindon.roberts@sydney.edu.au}.}
}

\date{\today}

\maketitle

\begin{abstract}
We consider trust-region methods for solving optimization problems where the objective is the sum of a smooth, nonconvex function and a nonsmooth, convex regularizer. We extend the global convergence theory of such methods to include worst-case complexity bounds in the case of unbounded model Hessian growth, and introduce a new, simple nonsmooth trust-region subproblem solver based on combining several iterations of proximal gradient descent with a single projection into the trust region, which meets the sufficient descent requirements for algorithm convergence and has promising numerical results.
\end{abstract}

{\small

% 47H05 = Monotone operators and generalizations
% 49M37 = Numerical methods based on nonlinear programming
% 65K05 = Numerical mathematical programming methods
% 65K10 = Numerical optimization and variational techniques
% 90C30 = Nonlinear programming

}

\section{Introduction}

In this work, we consider trust-region methods for solving nonconvex, nonsmooth optimization problems of the form
\begin{align}
    \min_{x\in\RR^d} F(x) :=f(x) + h(x), \label{eq_composite}
\end{align}
where $f:\RR^d\to\RR$ is smooth (at least $C^1$) but nonconvex, and $h:\RR^d\to(-\infty,+\infty]$ is a proper, lower semicontinuous and (possibly nonsmooth) convex function.
The general form \eqref{eq_composite} encompasses convex-constrained optimization (where $h$ is the indicator function for the constraint set) \cite{Hough2022}, and problems from areas such as optimal control of PDEs \cite{Herzog2012} and imaging \cite{Ehrhardt2021}.
We assume that access to $f$ is available through at least zeroth- and first-order oracles (i.e.~at least $f$ and $\grad f$ are available to an algorithm) and access to $h$ is available through direct evaluations and its proximity operator.

Most commonly, particularly in data science \cite{Wright2022}, first-order methods such as the proximal gradient method or FISTA \cite{Beck2017} are used to solve \eqref{eq_composite}.
In contrast, trust-region methods are typically second-order, using either $\grad^2 f$ (if available) or a suitable quasi-Newton approximation such as symmetric rank-1 (SR1), and are well designed to exploit potential negative curvature in $f$ \cite{nocedal2006numerical}.

In this work, we build on a recent trust-region method for solving \eqref{eq_composite} \cite{Baraldi2023}.
The convergence analysis in \cite{Baraldi2023} allows for the model Hessians to grow unboundedly across iterations $k$, at a rate $O(k^t)$ as $k\to\infty$ for $t\in[0,1]$.
The uniformly bounded case $t=0$ is the most widely used assumption for theoretical results (e.g.~\cite{Cartis2022}), but the case $t=1$ can arise from using SR1 Hessian approximations \cite[Chapter 8.4.1.2]{Conn2000}, with convergence results in the smooth case (i.e.~$h\equiv 0$ in \eqref{eq_composite}) derived in \cite{Powell1984}; the algorithm does not converge if $t>1$ \cite{Toint1988}.
By extending very recent results for smooth problems \cite{Diouane2024}, we broaden the convergence analysis from \cite{Baraldi2023} and show worst-case complexity bounds 
that match those of the smooth case.
We also introduce a novel trust-region subproblem solver suitable for \eqref{eq_composite} based on a projected proximal gradient (PPG) iteration, and show that it meets the sufficient decrease requirements for the main trust-region algorithm.
Our numerical results show that our PPG method can outperform existing subproblem solvers \cite{Baraldi2023a}.

\subsection{Existing Works}

In the special case where $h$ is the indicator function for a convex constraint set, trust-region methods have been available for many years \cite{Conn1988,Conn1993}.
These methods achieve a worst-case complexity of $O(\epsilon^{-2})$ iterations to find an $\epsilon$-approximate first-order critical point, although using cubic regularization methods can improve this to $O(\epsilon^{-3/2})$ \cite{Cartis2012b,Cartis2022}, both in the case of uniformly bounded Hessians.
Similarly, we note there are trust-region methods for general nonsmooth $F$, not exploiting the $F=f+h$ structure, such as \cite{Qi1994,DennisJr.1995,Chen2013a,Garmanjani2016}, and \cite{Chen2023} considers $F=f+h$ but builds a smooth approximation to $F$ directly at each iteration.
Alternative composite formulations have also been studied, such as $F=f+h\circ c$ for $c$ smooth and nonconvex (and vector-valued) \cite{Grapiglia2016} and $F=g\circ c + h$ for nonsmooth convex $g,h$ and smooth nonconvex $c$ \cite{Burke2019,Burke2020}, or \eqref{eq_composite} restricted to a smooth manifold \cite{Zhao2024}.

For the specific case of \eqref{eq_composite}, proximal Newton methods extend the proximal gradient method to include second-order information about $f$ by modifying the form of the proximity operator (potentially making it significantly harder to evaluate).
This has allowed the development of second-order linesearch methods for \eqref{eq_composite} in the case of $f$ both convex \cite{Lee2014} and nonconvex \cite{Lee2019a,Kanzow2021}.

Trust-region methods for \eqref{eq_composite} but where $h$ is nonconvex and nonsmooth are considered in \cite{Aravkin2022,Aravkin2022a,Leconte2023}.
These methods achieve worst-case complexity bounds of $O(\epsilon^{-2})$ iterations for uniformly bounded Hessians, and develop several subproblem solvers tailored to specific choices of regularizer $h$.
By contrast, \cite{Kim2010} considers the case where $h$ and $f$ are both convex, yielding $O(\epsilon^{-1})$ worst-case complexity in this more restrictive setting.
We also note \cite{Liu2024a} considers the case where $f$ is nonconvex but the local quadratic approximations used in the trust-region subproblem are convex (such as in some formulations of nonlinear least-squares problems), and the subproblem is solved with accelerated first-order methods.

The most relevant works for our results are \cite{Baraldi2023,Baraldi2023a,Baraldi2024}.
These consider the same problem \eqref{eq_composite}, albeit in an infinite-dimensional setting.
In \cite{Baraldi2023}, the general algorithmic framework is given, including sufficient decrease conditions for the trust-region subproblem.
They show global convergence to first-order critical points under the assumption of potentially unbounded Hessians, specifically $\sum_{k=0}^{\infty} (1 + \max_{j\in\{0,\dots,k\}} \|H_j\|)^{-1} = \infty$, where $H_k$ is the model Hessian at iteration $k$, essentially assuming a growth rate of $\|H_k\| = O(k)$.
Global convergence of smooth trust-region methods under this Hessian growth condition were shown in \cite{Powell1984} in the unconstrained case, and \cite{Toint1988} for the convex-constrained case.
The theory from \cite{Toint1988} is extended to the case of general $h$ in \cite{Baraldi2023}, to give global convergence, and $O(\epsilon^{-2})$ complexity in the case of uniformly bounded model Hessians; this is extended to local convergence rates in \cite{Baraldi2024} (with different local rates depending on the quality of Hessian approximation in the model for $f$).

A variety of trust-region subproblem solvers for \eqref{eq_composite} are also studied in \cite{Baraldi2023,Baraldi2023a}.
These include nonsmooth generalizations of the standard Cauchy point and dogleg method for the smooth ($h\equiv 0$) case and a nonlinear conjugate gradient method. 
However, from the numerical results in \cite{Baraldi2023a}, the most successful method developed in these works is the so-called spectral proximal gradient (SPG) method.
This method essentially performs proximal gradient iterations with some degree of linesearch to enforce both sufficient decrease and feasibility with respect to the trust region constraint.

Lastly, we note that both \cite{Kim2010} and \cite{Ouyang2024} consider trust-region subproblem solvers for \eqref{eq_composite} based on combining an iterative solver with a final projection step, which is similar to the approach in our new subproblem solver.
However, \cite{Kim2010} considers only the case where $f$ is convex, and \cite{Ouyang2024} uses a semismooth Newton formulation.

\subsection{Contributions}

We provide two main contributions in this work.

Firstly, we build on the theoretical analysis of the general nonsmooth trust-region method from \cite{Baraldi2023} to include worst-case complexity in the case of unbounded model Hessians.
Broadly speaking, if the model Hessians grow at a rate $\|H_k\|=O(k^t)$ for $t\in[0,1]$, then we prove convergence to an $\epsilon$-approximate first-order critical point after at most $O(\epsilon^{-2/(1-t)})$ iterations if $t\in[0,1)$, and $\tilde{O}(e^{-c\epsilon^{-2}})$ iterations (for some $c>0$) if $t=1$.
Our analysis is based on recent complexity results for the smooth case $h\equiv 0$ \cite{Diouane2024}; our worst-case complexity bounds in the nonsmooth case are identical to the smooth case.

Secondly, we introduce a new trust-region subproblem solver suitable for use in the main algorithm, which we call a {\em projected proximal gradient (PPG) method}.
Our approach is simpler than the SPG method from \cite{Baraldi2023,Baraldi2023a} in that we perform several iterations of proximal gradient descent on our model, and only once, at the end of our iterations, compute a projection to satisfy the trust region constraint.
We show that this method satisfies the sufficient decrease conditions required for global convergence (and complexity bounds) of the main trust-region algorithm.
We numerically compare our new PPG method against SPG on a large collection of $\ell_1$-regularized CUTEst \cite{Gould2015} problems, and show that PPG can outperform SPG (in the sense of the whole trust-region algorithm converging more quickly) when high accuracy solutions are desired.
We also observe that PPG is better able to make use of larger subproblem iteration budget if available.

\subsection{Organization and Notation}

In Section~\ref{sec_background}, we provide the relevant technical background for nonsmooth functions.
The main trust-region algorithm and its worst-case complexity in the unbounded Hessian case are given in Section~\ref{sec_complexity}.
In Section~\ref{sec_subproblem}, we introduce our new subproblem solver, and we present our numerical results in Section~\ref{sec_numerics}.

Throughout, we use $\|\cdot\|$ to represent the Euclidean norm in $\RR^d$ and the matrix 2-norm.

\section{Preliminaries} 
\label{sec_background}

Given $h\colon \RR^d\to (-\infty, +\infty]$, its \emph{domain} is $\dom h :=\{x\in \RR^d: h(x) <+\infty\}$ and its \emph{epigraph} is $\epi h :=\{(x, \rho)\in \RR^d\times \mathbb{R}: \rho \geq h(x)\}$. 
We recall that $h$ is \emph{proper} if $\dom h\neq \varnothing$, \emph{lower semicontinuous} if $\epi h$ is a closed set, and \emph{convex} if $\epi h$ is a convex set.

Let $h\colon \RR^d\to (-\infty, +\infty]$ be proper. The \emph{regular subdifferential} of $h$ at $x\in \RR^d$ is defined by
\begin{align}
    \partial h(x) :=\left\{u\in \RR^d:\; \liminf_{y\to\RR^d} \frac{h(y) -h(x) -u^\top(y -x)}{\|y -x\|}\geq 0\right\},
\end{align}
if $x\in\dom h$, and $\partial h(x) :=\varnothing$ otherwise.
When $h$ is convex, the regular subdifferential coincides with the \emph{convex subdifferential}, see, e.g., \cite[Theorem~1.93]{mordukhovich2006},
\begin{align}\label{eq:cvxsubdiff}
    \partial h(x) =\{u\in \RR^d:\; \forall y\in \RR^d,\; h(x) +u^\top(y -x)\leq h(y)\}.
\end{align}

\begin{proposition}\label{p:chainrule-Mor06}
Let $h\colon \RR^d\to (-\infty, +\infty]$ be proper and $f\colon \RR^d\to (-\infty, +\infty]$ be differentiable at $x\in\dom h$. Then
\begin{align}
\partial(f +h)(x) =\nabla f(x) +\partial h(x).
\end{align}    
\end{proposition}
\begin{proof}
See~{\cite[Proposition~1.107(i)]{mordukhovich2006}}.
\end{proof}

We recall that the \emph{proximity operator} of $h$ with parameter $\gamma>0$ at $x\in \RR^d$ is
\begin{align}
    \prox_{\gamma h}(x) :=\argmin_{z\in \RR^d} \left(h(z) +\frac{1}{2\gamma}\|z -x\|^2\right).
\end{align}

\begin{proposition} \label{p:prox}
Let $h\colon \RR^d\to (-\infty, +\infty]$ be a lower semicontinuous convex function. Then 
\begin{enumerate}
\item 
$\prox_{\gamma h}$ is single-valued, and $u =\prox_{\gamma h}(x)$ if and only if $x -u\in \gamma \partial h(u)$.
\item 
$\prox_{\gamma h}$ is nonexpansive, i.e., for all $x,y\in\RR^d$, $\|\prox_{\gamma h}(x) - \prox_{\gamma h}(y)\| \leq \|x-y\|$.
\end{enumerate}
\end{proposition}
\begin{proof}
See, e.g.~\cite[Theorems 6.39 \& 6.42]{Beck2017}.
\end{proof}

For problem \eqref{eq_composite}, we will assume $f$ is differentiable and we have access to the proximity operator for $h$ (see Assumption~\ref{ass_smoothness} below), and so we consider the first-order stationarity measure given by
\begin{align}
    \pi(x,\gamma) := \frac{1}{\gamma}\|\prox_{\gamma h}(x-\gamma \grad f(x))-x\|,
    \label{e:pigamma}
\end{align}
for any $x\in\RR^d$ and $\gamma>0$, which is the same measure considered in \cite{Baraldi2023a}.
When we consider our main algorithm and generate iterates $x_k$, we will use the shorthand $\pi_k(\gamma) := \pi(x_k, \gamma)$.

The function $\pi(x,\gamma)$ has several useful properties for a stationarity measure, such as $\pi(x,\gamma) \geq 0$ with $\pi(x,\gamma) = 0$ if and only if $0\in \grad f(x) + \partial h(x)$, and being continuous in both $x$ and $\gamma$ \cite[Proposition 1]{Baraldi2023a}.

\begin{proposition}
\label{p:pi_properties}
For any $x\in\RR^d$ and $0 < \gamma_1 \leq \gamma_2$,
    \begin{align}
        \pi(x,\gamma_2) \leq \pi(x,\gamma_1) \leq \frac{\gamma_2}{\gamma_1}\pi(x,\gamma_2).
    \end{align}    
\end{proposition}
\begin{proof}
     See~\cite[Theorem 10.9]{Beck2017}.
\end{proof}

For $\lambda\in \mathbb{R}$, a function $f\colon \RR^d\to (-\infty, +\infty]$ is said to be \emph{$\lambda$-convex} if for all $x, y\in \dom f$ and all $\alpha\in (0, 1)$,
\begin{align}
f((1 -\alpha)x +\alpha y) +\frac{\lambda}{2}\alpha(1 -\alpha)\|x -y\|^2\leq (1 -\alpha)f(x) +\alpha f(y).
\end{align}
It is well known that $f$ is $\lambda$-convex if and only if $f -\frac{\lambda}{2}\|\cdot\|^2$ is convex, see, e.g., \cite[Section~5]{Dao2019}. Then we say that $f$ is \emph{strongly convex} if $\lambda >0$, and \emph{weakly convex} if $\lambda <0$. 
More properties of $\lambda$-convex functions can be found in, e.g., \cite[Section~5]{Dao2019}.

\begin{example} \label{ex:lm-cvx}
Recall that for a symmetric matrix $H\in\RR^{d\times d}$, we have $H \succeq \lambda_{\min} I$ where $\lambda_{\min}$ is the minimum eigenvalue of $H$.
\begin{enumerate}
    \item The quadratic function $p\mapsto \frac{1}{2}p^\top Hp$ is $\lambda_{\min}$-convex. In addition, we note that
    $\|H\|=\max\{|\lambda|: \lambda \text{ is an eigenvalue of } H\}$. Thus,
    \begin{align}
        \|H\|\geq \lambda_{\min}.
    \end{align}
    \item If $f_1$ is $\lambda_1$-convex and $f_2$ is $\lambda_2$-convex, then $f_1+f_2$ is $(\lambda_1+\lambda_2)$-convex, see, e.g., \cite[Lemma~5.3]{Dao2019}. Thus, for any $x\in\RR^d$, $g\in\RR^d$, and any convex function $h$, the function
    \begin{align}
    m(p) := c + g^\top p + \frac{1}{2}p^\top Hp + h(x+p),
    \end{align}
    is also $\lambda_{\min}$-convex.
\end{enumerate}
\end{example}

\begin{proposition} 
\label{p:lmconvex}
Let $f\colon \RR^d\to (-\infty, +\infty]$ be proper, lower semicontinuous, and $\lambda $-convex. 
Then, for all $x, y\in \RR^d$ and all $u\in \partial f(x)$,
\begin{align}
    f(y) -f(x)\geq u^\top(y -x) +\frac{\lambda}{2}\|y -x\|^2.
\end{align}
\end{proposition}
\begin{proof}
This follows from \cite[Proposition~2.1(ii)]{Dao24}.    
\end{proof}

\section{Trust-Region Algorithm and Worst-Case Complexity} \label{sec_complexity}

We begin by outlining our main trust-region algorithm for solving \eqref{eq_composite}, in the case of Assumption~\ref{ass_smoothness}.
This algorithm is essentially the same as \cite[Algorithm 1]{Baraldi2023a}, but we will extend the analysis from global convergence to worst-case complexity, following the approach from \cite{Diouane2024}.
Our algorithm is designed to solve problems of the form \eqref{eq_composite} satisfying the following conditions.

\begin{assumption}
\label{ass_smoothness}
~
\begin{enumerate} %[label=(\alph*)]
\item 
The smooth component $f$ is differentiable with $\grad f$ being $L_{\grad f}$-Lipschitz continuous.
\item\label{ass_smoothness_h} 
The nonsmooth component $h$ is proper, lower semicontinuous and convex. 
\item 
The whole objective $F$ is bounded below by $F_{\textnormal{low}}$.
\end{enumerate}
\end{assumption}

Although not a formal assumption, our algorithm is designed using the proximity operator of $h$, and so we assume $\prox_{\gamma h}$ is practical to compute.
See \cite[Chapter 6.9]{Beck2017} for examples of such functions.

\subsection{Main Algorithm}

As is standard in trust-region methods \cite{Conn2000}, at each iteration $k$ (with current iterate $x_k$), we construct a local approximation to the full objective $F$ \eqref{eq_composite}.
In our case, we approximate the smooth part $f$ with a quadratic approximation and assume exact access to $h$, yielding an approximation
\begin{align}
    F(x_k+p) \approx m_k(p) := f(x_k) + \grad f(x_k)^\top p + \frac{1}{2}p^\top H_k p + h(x_k+p), \label{eq_model}
\end{align}
where $H_k\in\RR^{d\times d}$ is some symmetric matrix approximating the curvature of $f$, for example via symmetric rank-1 quasi-Newton updating \cite[Chapter 6.2]{nocedal2006numerical}.
We then compute a tentative step $p_k$ by approximately minimizing the model subject to a constraint on the size of the step, i.e., solve the trust-region subproblem
\begin{align}
    p_k \approx \argmin_{p\in\RR^d} m_k(p) \quad \text{s.t.} \quad \|p\| \leq \Delta_k, \label{eq_trs}
\end{align}
for some value $\Delta_k>0$ which is dynamically updated by the algorithm.
Lastly, we decide whether to accept the tentative new iterate $x_k+p_k$ by computing the ratio of the actual decrease in the objective from the step $p_k$ compared to the predicted decrease implied by the model \eqref{eq_rho}.
This full procedure is given in Algorithm~\ref{alg_nonsmooth_tr}.

\begin{algorithm}[htb]
\begin{algorithmic}[1]
\Require Starting point $x_0\in\dom F$ and trust-region radius $\Delta_0>0$. Algorithm parameters: scaling factors $0 < \gammadec < 1 < \gammainc$ and acceptance threshold $\eta\in (0,1)$.
\For{$k=0,1,2,\dots$}
    \State Build a local quadratic model $m_k$ \eqref{eq_model} satisfying Assumption~\ref{ass_unbounded_hess}.
    \State Solve the trust-region subproblem \eqref{eq_trs} to get a step $p_k$ satisfying Assumption~\ref{ass_cauchy_decrease}.
    \State Evaluate $F(x_k+p_k)$ and calculate the ratio 
    \begin{align}
        \rho_k := \frac{F(x_k) - F(x_k + p_k)}{m_k(0) - m_k(p_k)}. \label{eq_rho}
    \end{align}
    \If{$\rho_k \geq \eta$} 
        \State \textit{(Successful iteration)} Set $x_{k+1}=x_k+p_k$ and $\Delta_{k+1}=\gammainc\Delta_k$.
    \Else
        \State \textit{(Unsuccessful iteration)} Set $x_{k+1}=x_k$ and $\Delta_{k+1}=\gammadec\Delta_k$. 
    \EndIf
\EndFor
\end{algorithmic}
\caption{Nonsmooth trust-region method for \eqref{eq_composite}.}
\label{alg_nonsmooth_tr}
\end{algorithm}

As in \cite{Baraldi2023,Baraldi2023a}, we assume that our model Hessian $H_k$ may grow unboundedly.
Specifically, we follow \cite[Model Assumption 4.1]{Diouane2024} and (broadly speaking) assume that $\|H_k\| = O(k^t)$ for some $t\in[0,1]$.
The case $t=0$ corresponds to the more common assumption of uniformly bounded Hessians, but $t=1$ can occur when $H_k$ is generated by symmetric rank-1 quasi-Newton updating \cite[Chapter 8.4.1.2]{Conn2000}.
If $t>1$, the algorithm can be shown to not be globally convergent \cite{Toint1988}.

\begin{assumption} \label{ass_unbounded_hess}
    There exist $\mu>0$ and $t\in[0,1]$ such that, for all $k\in \NN$, 
    \begin{align}
    \max_{j\in\{0,\dots,k\}} \|H_j\| \leq \mu(1+k^t). 
    \end{align}
\end{assumption}

Lastly, we again follow \cite{Baraldi2023a} and require the following sufficient decrease condition from the (approximate) trust-region step \eqref{eq_trs}.

\begin{assumption} \label{ass_cauchy_decrease}
    There exists $\kappa_p>0$ and $\gamma_{\max}>0$ (both independent of $k$) such that for all iterations $k\in \NN$, the step $p_k$ satisfies
    \begin{align}
        m_k(0) - m_k(p_k) \geq \kappa_p \pi_k(\gamma_{\max}) \min\left\{\Delta_k, \frac{\pi_k(\gamma_{\max})}{1+\|H_k\|} \right\},
    \end{align}
    where $\pi_k(\gamma) := \pi(x_k,\gamma)$ is the first-order stationarity measure \eqref{e:pigamma} at $x_k$.
\end{assumption}

A key feature of this work is the new {\em projected proximal gradient subproblem solver} introduced in Section~\ref{sec_subproblem}.
In Corollary~\ref{cor_ppg_multipleN}, we will show that this solver generates the steps $p_k$'s that satisfy Assumption~\ref{ass_cauchy_decrease}.
Alternatively, \cite{Baraldi2023a} provides several suitable subproblem solvers.

The global convergence of Algorithm~\ref{alg_nonsmooth_tr} has already been established in the following result.

\begin{theorem} \label{thm_global_conv_old}
    Suppose Assumptions~\ref{ass_smoothness}, \ref{ass_unbounded_hess} and \ref{ass_cauchy_decrease} hold.
    Then 
\begin{align}
\liminf_{k\to\infty} \pi_k(\gamma_{\max}) = 0.
\end{align}
\end{theorem}
\begin{proof}
    This is \cite[Theorem 1]{Baraldi2023a}, which is a small extension of \cite[Theorem 3]{Baraldi2023}.
\end{proof}

\subsection{Worst-Case Complexity}

We now extend the analysis of Algorithm~\ref{alg_nonsmooth_tr} beyond the global convergence in Theorem~\ref{thm_global_conv_old} and consider its worst-case complexity.
That is, we wish to bound the number of iterations it takes to drive the stationarity measure $\pi_k(\gamma_{\max})$ below some desired level $\epsilon>0$.
The analysis in this section follows \cite{Diouane2024}, which considers only the smooth case $h\equiv 0$.

First, we bound the error in our model \eqref{eq_model}. Denote
\begin{align}
c_1 := \frac{1}{2}\max\{L_{\grad f}, 1\}. \label{e:c1}
\end{align}
\begin{lemma} \label{lem_fully_linear}
    Suppose Assumption~\ref{ass_smoothness} holds. Then, for all $\in \NN$, the trust-region step $p_k$ satisfies
    \begin{align}
        |F(x_k+p_k) - m_k(p_k)| \leq c_1 (1 + \|H_k\|) \Delta_k^2,
    \end{align}
    where $c_1$ is given by \eqref{e:c1}. 
\end{lemma}
\begin{proof}
    Since $\grad f$ is $L_{\grad f}$-Lipschitz continuous, we have the standard bound (e.g.~\cite[Appendix A]{nocedal2006numerical})
    \begin{align}
        |f(x_k+p_k) - f(x_k) - \grad f(x_k)^\top p_k| \leq \frac{1}{2} L_{\grad f}\|p_k\|^2.
    \end{align}
    Then, by the definitions of $F$ and $m_k$,
    \begin{align}
        |F(x_k+p_k) - m_k(p_k)| &= |f(x_k + p_k) + h(x_k+p_k) \nonumber \\
        & \qquad - f(x_k) - \grad f(x_k)^\top p_k - \frac{1}{2}p_k^\top H_k p_k - h(x_k+p_k)|, \\
        &\leq \frac{1}{2}L_{\grad f}\|p_k\|^2 + \frac{1}{2}\|H_k\| \cdot \|p_k\|^2, \\
        &\leq c_1 \|p_k\|^2 + c_1 \|H_k\| \cdot \|p_k\|^2.
    \end{align}
    The result then follows from $\|p_k\| \leq \Delta_k$.
\end{proof}

An important quantity to track in our analysis is
\begin{align}
    a_k := \Delta_k \cdot \frac{1 + \max_{j\leq k} \|H_j\|}{\min_{j\leq k} \pi_j(\gamma_{\max})}. \label{eq_ak_defn}
\end{align}
We have the following.

\begin{lemma} \label{lem_amin}
    Suppose Assumptions~\ref{ass_smoothness} and \ref{ass_cauchy_decrease} hold. 
    Then, for all $k\in \NN$, 
    \begin{align}
        a_k \geq a_{\min} := \min\left\{a_0, \gammadec, \frac{\gammadec \kappa_p (1-\eta)}{c_1}\right\},
        \label{e:a_min}
    \end{align}
where $a_k$ and $c_1$ are from \eqref{eq_ak_defn} and \eqref{e:c1}, respectively.
\end{lemma}
\begin{proof}
    We first note that $a_k/\Delta_k$ is non-decreasing by definition, and so
    \begin{align}
        \frac{a_{k+1}}{\Delta_{k+1}} \geq \frac{a_k}{\Delta_k}. \label{eq_amin_tmp1}
    \end{align}
    We will show the result by induction. The case $k=0$ is trivial by the definition of $a_{\min}$, so now suppose $a_k \geq a_{\min}$ for some $k$.

    First, if $a_k \geq \min\{1, \kappa_p (1-\eta)/c_1\}$, then by the updating mechanism for $\Delta_k$ (see~Algorithm~\ref{alg_nonsmooth_tr}) and \eqref{eq_amin_tmp1} we have
    \begin{align}
        a_{k+1} = \frac{a_{k+1}}{\Delta_{k+1}} \cdot \Delta_{k+1} \geq \frac{a_k}{\Delta_k} \cdot \gammadec \Delta_k = \gammadec a_k \geq \gammadec \min\left\{1, \frac{\kappa_p (1-\eta)}{c_1}\right\} \geq a_{\min},
    \end{align}
    and we are done.

    Now suppose that $a_k < \min\{1, \kappa_p (1-\eta)/c_1\}$.
    Then since $m_k(0) = F(x_k)$ by the definition of $m_k$, we have
    \begin{align}
        |\rho_k-1| = \frac{|F(x_k+p_k) - m_k(p_k)|}{m_k(0) - m_k(p_k)} \leq \frac{c_1 (1 + \|H_k\|) \Delta_k^2}{\kappa_p \pi_k(\gamma_{\max}) \min\left\{\Delta_k, \frac{\pi_k(\gamma_{\max})}{1+\|H_k\|}\right\}},
    \end{align}
    using Lemma~\ref{lem_fully_linear} and Assumption~\ref{ass_cauchy_decrease}.
    Hence
    \begin{align}
        |\rho_k-1| &\leq \frac{c_1 \left(1 + \max_{j\leq k} \|H_j\|\right) \Delta_k^2}{\kappa_p \left[\min_{j\leq k} \pi_j(\gamma_{\max})\right] \min\left\{\Delta_k, \frac{\min_{j\leq k} \pi_j(\gamma_{\max})}{1 + \max_{j\leq k} \|H_j\|}\right\}}, \\
        &= \frac{c_1 a_k \Delta_k}{\kappa_p \min\left\{\Delta_k, \frac{\Delta_k}{a_k}\right\}}, \\
        &= \frac{c_1}{\kappa_p} a_k < 1-\eta,
    \end{align}
    using the assumption $a_k < 1$ to get $\min\{\Delta_k, \Delta_k/a_k\} = \Delta_k$ for the last equality.
    Thus we have $\rho_k \geq \eta$, so iteration $k$ is successful.
    By the trust-region updating mechanism, $\Delta_{k+1} \geq \Delta_k$, and so using \eqref{eq_amin_tmp1} we get
    \begin{align}
        a_{k+1} = \frac{a_{k+1}}{\Delta_{k+1}} \cdot \Delta_{k+1} \geq \frac{a_k}{\Delta_k} \cdot \Delta_k = a_k \geq a_{\min},
    \end{align}
    and we are done.
\end{proof}

\begin{lemma} \label{lem_suff_decrease}
    Suppose Assumptions~\ref{ass_smoothness} and \ref{ass_cauchy_decrease} hold. 
    Then, for all $\in \NN$,
    \begin{align}
        m_k(0) - m_k(p_k) \geq \kappa_p a_{\min} \cdot \frac{\min_{j\leq k} \pi_j(\gamma_{\max})^2}{1 + \max_{j\leq k} \|H_j\|},
    \end{align}
where $a_{\min}$ is defined in \eqref{e:a_min}.
\end{lemma}
\begin{proof}
    From Assumption~\ref{ass_cauchy_decrease} and the definition \eqref{eq_ak_defn} of $a_k$, we have
    \begin{align}
        m_k(0) - m_k(p_k) &\geq \kappa_p \pi_k(\gamma_{\max}) \min\left\{\Delta_k, \frac{\pi_k(\gamma_{\max})}{1 + \|H_k\|}\right\}, \\
        &\geq \kappa_p \left[\min_{j\leq k} \pi_j(\gamma_{\max})\right] \min\left\{\frac{a_k \cdot \min_{j\leq k} \pi_j(\gamma_{\max})}{1 + \max_{j\leq k} \|H_j\|}, \frac{\min_{j\leq k} \pi_j(\gamma_{\max})}{1 + \max_{j\leq k} \|H_j\|}\right\}, \\
        &= \kappa_p \cdot \min\{a_k,1\} \cdot \frac{\min_{j\leq k} \pi_j(\gamma_{\max})^2}{1 + \max_{j\leq k} \|H_j\|}.
    \end{align}
    The result then follows from $a_k \geq a_{\min}$ (Lemma~\ref{lem_amin}) and $1 > \gammadec \geq a_{\min}$.
\end{proof}

We also need the following two technical lemmas:
The first one is relevant to Assumption~\ref{ass_unbounded_hess}, and the second one will be used to bound the time before the first successful iteration.

\begin{lemma} \label{lem_technical_sum}
    Suppose $\mu>0$ and $t>0$.
    Then, for all $k_1 < k_2$,
    \begin{align}
        \sum_{k=k_1}^{k_2} \frac{1}{1 + \mu(1+(k+1)^t)} \geq \frac{(k_1+1)^t}{1+\mu(1+(k_1+1)^t)} \int_{k_1+1}^{k_2+2} \frac{1}{s^t} ds.
    \end{align}
\end{lemma}
\begin{proof}
    This is \cite[Lemma 5]{Diouane2024}.
\end{proof}

\begin{lemma} \label{lem_technical}
    Suppose $a_1,a_2>0$. Then there exists $k^*$ such that $k > k^*$ implies $k > a_1 + a_2 \sqrt{k}$, and moreover we have $k^*=\Theta(a_1)$ as $a_1\to\infty$.
\end{lemma}
\begin{proof}
    The inequality $k\geq a_1 + a_2\sqrt{k}$ holds provided that $\sqrt{k}$ is at least the larger root of $t^2-a_2 t - a_1$, i.e, \begin{align}
    k> \left(\frac{a_2+ \sqrt{a_2^2+4a_1}}{2}\right)^2
    =\frac{2a_1+a_2^2+a_2\sqrt{a_2^2+4a_1}}{2},
    \end{align}
    which gives $k^*=\Theta(a_1)$.
\end{proof}

Our first main result bounds the total time until the first successful iteration.

\begin{lemma} \label{lem_first_success}
    Suppose Assumptions~\ref{ass_smoothness}, \ref{ass_unbounded_hess} and \ref{ass_cauchy_decrease} hold.
    Then if $\pi_0(\gamma_{\max}) \geq \epsilon > 0$, then there is at least one successful iteration, and if $k_0$ is the first successful iteration, then $k_0 = O(\log(1/\epsilon))$ as $\epsilon\to 0$.
\end{lemma}
\begin{proof}
Suppose that the iterations $0,\dots,k_0-1$ are unsuccessful, so $x_0 = x_1 = \cdots = x_{k_0}$. Thus, $\pi_0(\gamma_{\max}) = \cdots = \pi_{k_0}(\gamma_{\max}) \geq \epsilon$ and $\Delta_{k_0} = \gammadec^{k_0} \Delta_0$.
    Applying Lemma~\ref{lem_amin} and Assumption~\ref{ass_unbounded_hess}, we have (for all $k$)
    \begin{align}
        \gammadec^{k_0} \Delta_0 = \Delta_{k_0} \geq \frac{\min_{j\leq k_0} \pi_j(\gamma_{\max})}{1 + \max_{j\leq k_0} \|H_j\|} \cdot a_{\min} \geq \frac{\epsilon}{1 + \mu(1+k_0^t)} \cdot a_{\min}. \label{eq_first_success_tmp1}
    \end{align}
    Equivalently,
    \begin{align}
        \log(1+\mu(1+k_0^t)) \geq \log\left(\frac{\epsilon a_{\min}}{\Delta_0}\right) + k_0 \log(\gammadec^{-1}).
    \end{align}
    Using the identity $\log(x) \leq \frac{2}{e}\sqrt{x}$ for all $x>0$ \cite[Section 3.6.1]{Mitrinovic1970}, together with $t\leq 1$ we get
    \begin{align}
        \frac{2}{e}\sqrt{1+\mu+\mu k_0} \geq \log\left(\frac{\epsilon a_{\min}}{\Delta_0}\right) + k_0 \log(\gammadec^{-1}).
    \end{align}
    Now using the identity $\sqrt{x+y} \leq \sqrt{x} + \sqrt{y}$ for $x,y\geq 0$, we have
    \begin{align}
        \frac{2\sqrt{1+\mu}}{e} + \frac{2\sqrt{\mu}}{e} \sqrt{k_0} \geq \log\left(\frac{\epsilon a_{\min}}{\Delta_0}\right) + k_0 \log(\gammadec^{-1}),
    \end{align}
    or
    \begin{align}
        k_0 \leq  \frac{1}{\log(\gammadec^{-1})}\left[\frac{2\sqrt{1+\mu}}{e} + \log\left(\frac{\Delta_0}{\epsilon a_{\min}}\right)\right] + \frac{2\sqrt{\mu}}{e \log(\gammadec^{-1})} \sqrt{k_0}.
    \end{align}
    From Lemma~\ref{lem_technical}, we get that $k_0$ is finite, and $k_0 \leq O(\log(1/\epsilon))$ as $\epsilon\to 0$.
\end{proof}

\begin{remark}
If the model Hessian is formed via quasi-Newton updating, it is common that it is only updated on successful iterations.
In that setting, if $k_0$ is the first successful iteration, then we have $H_0=H_1=\cdots=H_{k_0}$. 
So, \eqref{eq_first_success_tmp1} in the proof of Lemma~\ref{lem_first_success} can be replaced by $\gammadec^{k_0} \Delta_0 \geq \frac{\epsilon}{1 + \|H_0\|}a_{\min}$, from which the final result $k_0 = O(\log(1/\epsilon))$ follows immediately.
\end{remark}

Now consider a desired accuracy level $\epsilon$, and define $k_{\epsilon}$ to be the first iteration with $\pi_{k_{\epsilon}}(\gamma_{\max}) < \epsilon$ (or $k_{\epsilon}=\infty$ if this never occurs).

If $\pi_0(\gamma_{\max}) < \epsilon$, then we have $k_{\epsilon}=0$ and no further effort is required.
Hence we will assume our desired accuracy level satisfies $\epsilon \leq \pi_0(\gamma_{\max})$.
In this case, from Lemma~\ref{lem_first_success} we have a first successful iteration $k_0$ with $k_0=O(\log(1/\epsilon))$.
Since $x_k$ is only changed on successful iterations, we have $\pi_k(\gamma_{\max}) = \pi_0(\gamma_{\max})$ for all $k\leq k_0$, and so we must have $k_{\epsilon} > k_0$.

Next, we assume that
\begin{align}
\text{$\tau$ is a positive integer such that $\gammainc \gammadec^{\tau-1} < 1$},\label{e:tau}
\end{align}
which must exist due to $\gammadec < 1$,
and define the following sets:
\begin{itemize}
    \item $\mathcal{S} := \{j \in\{0,1,2,\dots\} : \rho_j \geq \eta\}$ is the set of all successful iterations
    \item $\mathcal{S}_k := \{j \in \{0,1,2,\dots,k\} : \rho_j \geq \eta\}=\{j \in \{k_0,\dots,k\} : \rho_j \geq \eta\}$ is the set of all successful iterations up to iteration $k$, with $\mathcal{S}(\epsilon) := \mathcal{S}_{k_{\epsilon}}$. 
    \item $\mathcal{U}_k := \{j \in \{0,1,2,\dots,k\} : \rho_j < \eta\}$ is the set of all unsuccessful iterations up to iteration $k$, with $\mathcal{U}(\epsilon) := \mathcal{U}_{k_{\epsilon}}$.
    \item For $k\geq k_0$, define $\mathcal{V}_k := \{j \in \{k_0,\dots,k\} : |\mathcal{S}_j| \geq j/\tau\}$ is the set of iterations for which we have had a relatively high fraction of successful iterations in the history to that point.\footnote{For example, if $\tau=3$ and $j\in \mathcal{V}_k$, then at least $1/3$ of the iterations $0, \dots, j$ were successful. We may choose $\tau=3$ if for example we have the common parameter choices $\gammainc=2$ and $\gammadec=0.5$.} 
    \item For $k\geq k_0$, define $\mathcal{W}_k := \{j \in \{k_0,\dots,k\} : |\mathcal{S}_j| < j/\tau\}$ is the set of iterations for which we have had a relatively low fraction of successful iterations, $\mathcal{W}_k = \{k_0,\dots,k\}\setminus \mathcal{V}_k$.
\end{itemize}

\begin{lemma} \label{lem_sum_split}
    Let $r_k$ be any strictly positive, non-decreasing sequence, and we have at least one successful iteration (i.e.~$k_0$ exists).
    Then, for all $k\geq k_0$,
    \begin{align}
        \tau \sum_{j\in \mathcal{S}_k} \frac{1}{r_j} \geq \sum_{j\in \mathcal{V}_k} \frac{1}{r_j} = \sum_{j=k_0}^{k}  \frac{1}{r_j} - \sum_{j\in \mathcal{W}_k} \frac{1}{r_j},
    \end{align}
    using the convention $\sum_{j\in\varnothing} \frac{1}{r_j} = 0$.
\end{lemma}
\begin{proof}
    This is \cite[Lemma 7]{Diouane2024}.
\end{proof}

\begin{lemma} \label{lem_sum_bound}
    Suppose Assumptions~\ref{ass_smoothness}, \ref{ass_unbounded_hess} and \ref{ass_cauchy_decrease} hold, and we have at least one successful iteration (i.e.~$k_0$ exists).
    Then for any $\epsilon>0$,
    \begin{align}
        \sum_{k\in\mathcal{W}_{k_{\epsilon}}} \frac{1}{1 + \mu(1+k^t)} \leq \frac{\Delta_0 \xi}{\epsilon a_{\min}},
    \end{align}
    again with the convention $\sum_{k\in\varnothing} a_k = 0$, and where 
    \begin{align}
        \xi := \sum_{k\in\mathbb{N}} (\gammainc \gammadec^{\tau-1})^{k/\tau} < \infty.
    \end{align}
\end{lemma}
\begin{proof}
    First, note that $\xi<\infty$ since it is a geometric series $\sum_{k\in\mathbb{N}} \delta^k$ for $\delta=(\gammainc \gammadec^{\tau-1})^{1/\tau}$ with $\delta<1$ by the definition of $\tau$ in \eqref{e:tau}.

    If $\mathcal{W}_{k_{\epsilon}}=\varnothing$ the result is trivial.
    Otherwise, fix $k\in\mathcal{W}_{k_{\epsilon}}$
    and let $r_k = 1 + \mu(1+k^t)$. We have in particular that $k\leq k_{\epsilon}$, which leads to
\begin{equation*}
\epsilon<\min_{j\leq k} \pi_j(\gamma_{\max}).
\end{equation*}
Together with Assumption~\ref{ass_unbounded_hess}, Lemma~\ref{lem_amin}, \eqref{eq_ak_defn}, and the update mechanism for $\Delta_k$, we obtain
    \begin{align}
        \frac{a_{\min} \epsilon}{r_k} \leq a_k \frac{\min_{j\leq k} \pi_j(\gamma_{\max})}{1 + \max_{j\leq k} \|H_j\|} = \Delta_k = \gammainc^{|\mathcal{S}_k|} \gammadec^{k-|\mathcal{S}_k|} \Delta_0.
    \end{align}
    Since $k\in\mathcal{W}_{k_{\epsilon}}$, we have $k > \tau|\mathcal{S}_k|$, and so
    \begin{align}
        \frac{a_{\min} \epsilon}{r_k} \leq \gammainc^{k/\tau} \gammadec^{k-k/\tau} \Delta_0 = (\gammainc \gammadec^{\tau-1})^{k/\tau} \Delta_0.
    \end{align}
    Summing over all $k\in\mathcal{W}_{k_{\epsilon}}$ then gives
    \begin{align}
        a_{\min} \epsilon \sum_{k\in\mathcal{W}_{k_{\epsilon}}} \frac{1}{r_k} \leq \sum_{k\in\mathcal{W}_{k_{\epsilon}}} (\gammainc \gammadec^{\tau-1})^{k/\tau} \Delta_0 \leq \sum_{k\in\mathbb{N}} (\gammainc \gammadec^{\tau-1})^{k/\tau} \Delta_0,
    \end{align}
    and we are done.
\end{proof}

We can now state our main result on the worst-case complexity.

\begin{theorem} \label{thm_wcc_final}
    Suppose Assumptions~\ref{ass_smoothness}, \ref{ass_unbounded_hess} and \ref{ass_cauchy_decrease} hold.
    Then, for any $0 < \epsilon \leq \pi_0(\gamma_{\max})$, if $0 \leq t < 1$ we have
    \begin{align}
        k_{\epsilon} \leq \left[(1-t)(c_2 \epsilon^{-2} + c_3 \epsilon^{-1})\frac{1+\mu(1+k_0^t)}{1+k_0^t} + (k_0+1)^{1-t}\right]^{1/(1-t)} - 2, \label{eq_wcc_final_1}
    \end{align}
    and if $t=1$ we have
    \begin{align}
        k_{\epsilon} \leq \exp\left((c_2 \epsilon^{-2} + c_3 \epsilon^{-1})\frac{1+\mu(1+k_0^t)}{1+k_0^t}\right)(k_0+1) - 2, \label{eq_wcc_final_2}
    \end{align}
    where in both bounds we use the constants
    \begin{align}
        c_2 := \frac{\tau [F(x_0) - F_{\textnormal{low}}]}{\eta \kappa a_{\min}} \qquad \text{and} \qquad c_3 := \frac{\Delta_0 \xi}{a_{\min}}.
    \end{align}
    If instead $\epsilon > \pi_0(\gamma_{\max})$, then $k_{\epsilon}=0$.
\end{theorem}
\begin{proof}
    If $\epsilon > \pi_0(\gamma_{\max})$, the result is trivial. 
    Otherwise, from $\epsilon \leq \pi_0(\gamma_{\max})$ and Lemma~\ref{lem_first_success}, we know that there must be a first successful iteration (i.e.~$k_0$ exists).

    For successful iterations $k\in\mathcal{S}(\epsilon)$, we have, using Lemma~\ref{lem_suff_decrease},
    \begin{align}
        F(x_k) - F(x_k+p_k) &\geq \eta (m_k(0) - m_k(p_k)), \\
        &\geq \eta \kappa a_{\min} \cdot \frac{\min_{j\leq k} \pi_j(\gamma_{\max})^2}{1 + \max_{j\leq k} \|H_j\|}, \\
        &\geq \eta \kappa a_{\min} \cdot \frac{\epsilon^2}{1 + \mu(1+k^t)}.
    \end{align}
    Summing over all $k\in\mathcal{S}(\epsilon)$ (and noting that $x_{k+1}=x_k$ for $k\notin\mathcal{S}(\epsilon)$), we have
    \begin{align}
        F(x_0) - F_{\textnormal{low}} \geq \sum_{k\in\mathcal{S}(\epsilon)} F(x_k) - F(x_k + p_k) \geq \eta \kappa a_{\min} \epsilon^2 \sum_{k\in\mathcal{S}(\epsilon)} \frac{1}{1 + \mu(1+k^t)}.
    \end{align}
    Define $r_k = 1 + \mu(1+k^t)$, which is strictly positive and increasing.
    Hence Lemmas~\ref{lem_sum_split} and \ref{lem_sum_bound} give
    \begin{align}
        F(x_0) - F_{\textnormal{low}} &\geq \frac{\eta \kappa a_{\min} \epsilon^2}{\tau} \left[\sum_{j=k_0}^{k_{\epsilon}} \frac{1}{r_j} - \sum_{j\in\mathcal{W}_{k_{\epsilon}}} \frac{1}{r_j}\right], \\
        &\geq \frac{\eta \kappa a_{\min} \epsilon^2}{\tau} \left[\sum_{j=k_0}^{k_{\epsilon}} \frac{1}{r_j} - \frac{\Delta_0 \xi}{\epsilon a_{\min}}\right].
    \end{align}
    Finally, Lemma~\ref{lem_technical_sum} gives
    \begin{align}
        F(x_0) - F_{\textnormal{low}} &\geq \frac{\eta \kappa a_{\min} \epsilon^2}{\tau} \left[\frac{(1+k_0^t)}{r_{k_0}} \int_{k_0+1}^{k_{\epsilon}+2} \frac{1}{s^t} ds - \frac{\Delta_0 \xi}{\epsilon a_{\min}}\right],
    \end{align}
    or equivalently
    \begin{align}
        \int_{k_0+1}^{k_{\epsilon}+2} \frac{1}{s^t} ds \leq \left(c_2 \epsilon^{-2} + c_3 \epsilon^{-1}\right) \frac{r_{k_0}}{1+k_0^t}.
    \end{align}
    We get the final result from evaluating the integral,
    \begin{align}
        \int_{k_0+1}^{k_{\epsilon}+2} \frac{1}{s^t} ds = \begin{cases} \frac{(k_{\epsilon}+2)^{1-t} - (k_0+1)^{1-t}}{1-t}, & \text{if $0 \leq t < 1$}, \\ \log(k_{\epsilon}+2)-\log(k_0+1), & \text{if $t=1$}, \end{cases}
    \end{align}
    and rearranging for $k_{\epsilon}$.
\end{proof}

In summary, the worst-case complexity bounds are the same as the smooth case \cite{Diouane2024}, as summarized below.
In the case of uniformly bounded Hessians ($t=0$), we recover the standard $O(\epsilon^{-2})$ worst-case iteration complexity for smooth trust-region methods (e.g.~\cite[Theorem 2.3.7]{Cartis2022}).

\begin{corollary}
    Suppose the assumptions of Theorem~\ref{thm_wcc_final} hold.
    If $0\leq t <1$, then $k_{\epsilon} = O(\epsilon^{-2/(1-t)})$, and if $t=1$, then $k_{\epsilon} = \tilde{O}(e^{c\epsilon^{-2}})$ for some $c>0$, where $\tilde{O}(\cdot)$ hides logarithmic factors of size $\log(1/\epsilon)$.
\end{corollary}
\begin{proof}
    This follows directly from \eqref{eq_wcc_final_1} and \eqref{eq_wcc_final_2}, and then using $k_0=O(\log(1/\epsilon))$ from Lemma~\ref{lem_first_success}.
\end{proof}

\section{Projected Proximal Gradient Subproblem Solver} \label{sec_subproblem}

In this section we introduce our new solver for the nonsmooth trust-region subproblem (c.f.~\eqref{eq_trs})
\begin{align}
    p^* \approx \argmin_{s\in\RR^d} m(p) := c + g^\top p + \frac{1}{2} p^\top H p + h(x+p), \quad \text{s.t.} \quad \|p\| \leq \Delta, \label{eq_trs_generic}
\end{align}
where here $x\in\RR^d$ is the current iterate, $c\in\RR$, $g\in\RR^d$ and $H\in\RR^{d\times d}$ (symmetric) form the current quadratic approximation for $f$ and $\Delta>0$ is the current trust-region radius.
Our approach is based on minimizing $m$ using the proximal gradient method \cite[Chapter 10.2]{Beck2017}, with the constraint $\|p\| \leq \Delta$ enforced at the end by projecting into the feasible region.

A full specification is given in Algorithm~\ref{alg_ppg}.
In particular, we note that the stepsize $\gamma$ for the proximal gradient part is a (to-be-specified) input, and to reduce the total effort we terminate the proximal gradient part early if we move too far outside the trust-region, $\|u_i - x\| > \mu_u \Delta_k$ (for $\mu_u > 1$).

\begin{algorithm}[htb]
\begin{algorithmic}[1]
\Statex Parameters: stepsize $\gamma>0$, maximum iterations $N\in\mathbb{N}$ and trust-region scaling $\mu_u \geq 1$.
\State Set $i=0$ and $u_0 = x$.
\While{$i < N$ and $\|u_i - x\| \leq \mu_u \Delta$}
    \State $u_{i+1} = \prox_{\gamma h}(u_i-\gamma g - \gamma H (u_i-x))$.
    \State $i = i + 1$.
\EndWhile
\State \textbf{return} $p^* = \frac{\Delta}{\max\{\Delta,\|u_i-x\|\}}(u_i-x)$.
\end{algorithmic}
\caption{Projected proximal gradient method for nonsmooth trust-region subproblem \eqref{eq_trs_generic}.}
\label{alg_ppg}
\end{algorithm}

First, we illustrate a lower bound for the decrease in $m$ achieved in the first iteration.

\begin{lemma} \label{l:prox descent}
    Suppose $h$ satisfies Assumption~\ref{ass_smoothness}\ref{ass_smoothness_h}, $\|H\| \leq L$ and $\lambda := \lambda_{\min}(H)$. 
    Then, for any iteration $i$ of Algorithm~\ref{alg_ppg},
    \begin{align}
        m(u_i-x) - m(u_{i+1}-x) \geq \left(\frac{1}{\gamma} - L + \frac{\lambda}{2}\right) \|u_{i+1}-u_i\|^2.
    \end{align}
\end{lemma}
\begin{proof}
From the iteration $u_{i+1} = \prox_{\gamma h}(u_i-\gamma g - \gamma H (u_i-x))$ and the optimality condition for the proximity operator (Proposition~\ref{p:prox}) we get
\begin{align}
    u_i - u_{i+1} - \gamma g - \gamma H (u_i - x) \in \gamma \partial h(u_{i+1}),
\end{align}
and so
\begin{align}
    \frac{1}{\gamma}(u_i -u_{i+1}) + H (u_{i+1} - u_i)\in g +H (u_{i+1}-x) + \partial h(u_{i+1}) = \partial m(u_{i+1} -x), \label{eq_prox_descent_tmp1}
\end{align}
where we have used Proposition~\ref{p:chainrule-Mor06} in the last equality.
By assumption, $m$ is $\lambda$-convex (see~Example~\ref{ex:lm-cvx}) and so from Proposition~\ref{p:lmconvex} we have
\begin{align}
    m(u_i-x) - m(u_{i+1} - x) &\geq \left(\frac{1}{\gamma}(u_i -u_{i+1}) +H(u_{i+1} -u_i)\right)^{\top} \left(u_i -u_{i+1}\right)\nonumber\\
    &\qquad\qquad +\frac{\lambda}{2}\|(u_i -x) -(u_{i+1} -x)\|^2 \\
    &=\frac{1}{\gamma}\|u_{i+1} -u_i\|^2 -(u_{i+1} -u_i)^\top H(u_{i+1} -u_i) +\frac{\lambda}{2}\|u_{i+1} -u_i\|^2 \\
    &\geq \left(\frac{1}{\gamma} -L +\frac{\lambda}{2}\right)\|u_{i+1} -u_i\|^2,
\end{align}
which completes the proof.
\end{proof}

We can now show that Algorithm~\ref{alg_ppg} achieves sufficient decrease in the model.

\begin{theorem} \label{p:decrease}
    Suppose $h$ satisfies Assumption~\ref{ass_smoothness}\ref{ass_smoothness_h}, $\|H\| \leq L$ and $\lambda := \lambda_{\min}(H)$. 
    Then the output $p^*$ of Algorithm~\ref{alg_ppg} satisfies
    \begin{align}
        m(0) - m(p^*) \geq \theta \|u_1 - x\| \min\left\{\Delta, \|u_1 - x\|\right\},
    \end{align}
    where
    \begin{align}
        \theta := \frac{1}{\gamma} - L + \frac{\lambda}{2} + \min\left\{0,\frac{\lambda R(\gamma L)^2}{2}\right\},
    \end{align}
    and where for any $t>0$ we define $R(t) := \sum_{i=0}^{N-1} (1+t)^i = \frac{(1+t)^N-1}{t}>1$.
\end{theorem}
\begin{proof}
    Suppose the loop in Algorithm~\ref{alg_ppg} terminates after $n\in\{1,\dots,N\}$ iterations (note that the termination conditions ensure at least one loop iteration is run).

    Since $\prox_{\gamma h}$ is nonexpansive (Proposition~\ref{p:prox}), for any $i\in\{0,\dots,n-2\}$,
    \begin{align}
        \|u_{i+2} - u_{i+1}\| \leq \|u_{i+1} - u_i - \gamma H(u_{i+1} - u_i)\| \leq (1+\gamma L) \|u_{i+1}-u_i\|, \label{eq_suff_decrease_tmp0}
    \end{align}
    and so by induction we have $\|u_{i+1}-u_i\| \leq (1+\gamma L)^i \|u_1-u_0\|$ for all $i\in\{0,\dots,n-1\}$.
    By definition of $R(t)$, this gives us
    \begin{align}
        \|u_n - x\| = \|u_n - u_0\| \leq \sum_{i=0}^{n-1} \|u_{i+1} - u_i\| \leq R(\gamma L) \|u_1-u_0\| = R(\gamma L) \|u_1-x\|, \label{eq_suff_decrease_tmp0a}
    \end{align}
    where the second inequality uses $n\leq N$.
    Next, we apply Lemma~\ref{l:prox descent} to obtain
    \begin{align}
        m(0) - m(u_n-x) &= \sum_{i=0}^{n-1} m(u_i-x) - m(u_{i+1}-x), \\
        &\geq \left(\frac{1}{\gamma} - L + \frac{\lambda}{2}\right) \sum_{i=0}^{n-1} \|u_{i+1}-u_i\|^2, \\
        &\geq \left(\frac{1}{\gamma} - L + \frac{\lambda}{2}\right) \|u_1-x\|^2, \label{eq_suff_decrease_tmp1}
    \end{align}
    where the last line uses $n\geq 1$, yielding $\sum_{i=0}^{n-1} \|u_{i+1}-u_i\|^2 \geq \|u_1-u_0\|^2 = \|u_1-x\|^2$.

    Now, define $\alpha := \frac{\Delta}{\max\{\Delta,\|u_n-x\|\}} \in (0,1]$, so the output of Algorithm~\ref{alg_ppg} is $p^* = \alpha (u_n-x)$.
    Writing $p^* = (1-\alpha) 0 + \alpha(u_n-x)$ and using the $\lambda$-convexity of $m$ (Proposition~\ref{p:lmconvex}) we get
    \begin{align}
        (1-\alpha) m(0) + \alpha m(u_n-x) \geq m(p^*) + \frac{\lambda}{2} \alpha(1-\alpha) \|u_n-x\|^2,
    \end{align}
    and so from \eqref{eq_suff_decrease_tmp1} we have
    \begin{align}
        m(0) - m(p^*) &\geq \alpha (m(0) - m(u_n-x)) + \frac{\lambda}{2} \alpha(1-\alpha) \|u_n-x\|^2, \\
        &\geq \alpha \left(\frac{1}{\gamma} - L + \frac{\lambda}{2}\right) \|u_1-x\|^2 + \frac{\lambda}{2}\alpha (1-\alpha)\|u_n-x\|^2.
    \end{align}
    If $\lambda \geq 0$ then we use $\alpha \leq 1$ to conclude $\frac{\lambda(1-\alpha)}{2} \geq 0$, and if $\lambda<0$ then we use $\alpha>0$ to get $\frac{\lambda(1-\alpha)}{2} > \frac{\lambda}{2}$.
    Hence
    \begin{align}
        \frac{\lambda (1-\alpha)}{2} \geq \min\left\{0,\frac{\lambda}{2}\right\},
    \end{align}
    and so
    \begin{align}
        m(0) - m(p^*) \geq \alpha \left(\frac{1}{\gamma} - L + \frac{\lambda}{2}\right) \|u_1-x\|^2 + \alpha \min\left\{0,\frac{\lambda}{2}\right\} \|u_n-x\|^2.
    \end{align}
    We then apply \eqref{eq_suff_decrease_tmp0a} to get
    \begin{align}
        m(0) - m(p^*) &\geq \alpha \left(\frac{1}{\gamma} - L + \frac{\lambda}{2}\right) \|u_1-x\|^2\nonumber\\
        &\qquad + \alpha \min\left\{0,\frac{\lambda}{2}\right\} R(\gamma L)^2 \|u_1-x\|^2 = \alpha \theta \|u_1-x\|^2.
    \end{align}
    Lastly, if $\|u_1-x\| \leq \Delta$ then $\alpha=1$, so $\alpha\|u_1-x\| = \|u_1-x\| = \min\{\Delta, \|u_1-x\|\}$.
    Instead, if $\|u_1-x\| > \Delta$ then $\alpha=\frac{\Delta}{\|u_1-x\|}$ and so $\alpha \|u_1-x\| = \Delta = \min\{\Delta, \|u_1-x\|\}$.
    In either case, we get the desired result.
\end{proof}

From Lemma~\ref{p:decrease}, using the fact that $\lim_{t\to 0} R(t) = N$, we note that $\theta>0$ holds provided we pick the stepsize $\gamma>0$ to be sufficiently small (for any choice of total iterations $N$).
However, in light of Assumption~\ref{ass_unbounded_hess}, the value of $L$ may grow unboundedly over the iterations of the main trust-region algorithm (Algorithm~\ref{alg_nonsmooth_tr}).
The below result shows that regardless of the value of $L$, there is a range of values of $\gamma$ that achieve  sufficient decrease.

\begin{lemma} \label{lem_good_gamma}
    Suppose $\|H\| \leq L$ and $g = \grad f(x)$, and we run Algorithm~\ref{alg_ppg} with at most $N\geq 1$ iterations with stepsize $\gamma>0$.
    Then there exists $c^*>0$ and $\kappa_s>0$ such that, for all $\gamma\in [c^*/(10L), c^*/L]$,
    \begin{align}
        m(0) - m(p^*) \geq \kappa_s \pi(x,\gamma) \min\left\{\Delta, \frac{\pi(x,\gamma)}{L}\right\}. \label{eq_suff_decrease}
    \end{align}
    The constants $c^*$ and $\kappa_s$ depend on $N$, but not $\gamma$ or $\mu_u$.
\end{lemma}
\begin{proof}
    From Lemma~\ref{p:decrease} with the conservative bound $\lambda=-L$ and $\|u_1-x\|=\gamma \pi(x,\gamma)$ by the definition \eqref{e:pigamma}, we have
    \begin{align}
        m(0) - m(p^*) \geq \left(1 - \frac{3\gamma L}{2} - \frac{\gamma L R(\gamma L)^2}{2}\right) \pi(x,\gamma) \min\left\{\Delta, (\gamma L)\frac{\pi(x,\gamma)}{L}\right\}.
    \end{align}
    Now suppose that $\gamma = c/L$ for some $0<c<1$.
    Then
    \begin{align}
        m(0) - m(p^*) &\geq \left(1 - \frac{3c}{2} - \frac{c R(c)^2}{2}\right)  \pi(x,\gamma) \min\{1,c\} \min\left\{\Delta, \frac{\pi(x,\gamma)}{L}\right\}\nonumber\\
        &=c \left(1 - \frac{3c}{2} - \frac{c R(c)^2}{2}\right)  \pi(x,\gamma) \min\left\{\Delta, \frac{\pi(x,\gamma)}{L}\right\}
    \end{align}
    Furthermore, since $R(c) = \sum_{i=0}^{N-1} (1+c)^i \leq N (1+c)^{N-1}$, a sufficient condition for the result to hold is therefore to choose $\kappa_s$ such that
    \begin{align}
        c \left(1 - \frac{3c}{2} - \frac{c [N(1+c)^{N-1}]^2}{2}\right) \geq \kappa_s > 0. \label{eq_good_gamma_tmp1}
    \end{align}
    Since the factor inside the brackets approaches 1 as $c\to 0^{+}$, there exists $c^*$ such that, for all $c\in (0,c^*]$,
    \begin{align}
        1 - \frac{3c}{2} - \frac{c [N(1+c)^{N-1}]^2}{2} \geq \frac{1}{2}.
        \label{eq_good_gamma_tmp2}
    \end{align}
    Thus \eqref{eq_good_gamma_tmp1} and hence the final result holds with, for example, $\kappa_s = \frac{c^*}{20}$ and any $c\in [c^*/10, c^*]$.
\end{proof}

Our final result confirms that the sufficient decrease requirement Assumption~\ref{ass_cauchy_decrease} can be achieved by running Algorithm~\ref{alg_ppg} with decreasing choices of $\gamma$, regardless of the size of $\|H\|$.

\begin{corollary} \label{cor_ppg_multipleN}
    Suppose the assumptions of Lemma~\ref{lem_good_gamma} hold.
    If we run Algorithm~\ref{alg_ppg} with stepsize choices $\gamma_j = \alpha^j \gamma_0$ for some $\alpha \in (0.1, 1)$ and $j=0,1,2,\dots$, with $\gamma_0 \leq \gamma_{\max}$, then there is a finite $j$ (possibly dependent on $L$ and $N$) such that
    \begin{align}
        m(0) - m(p^*) \geq \kappa_s \pi(x,\gamma_{\max}) \min\left\{\Delta, \frac{\pi(x,\gamma_{\max})}{L}\right\},
    \end{align}
    for some $\kappa_s>0$ depending on $N$ but independent of $L$.
\end{corollary}
\begin{proof}
    Since $\alpha\in (0.1, 1)$, there must be at least one $j$ for which $\gamma_j \in [c^*/(10L), c^*/L]$, and so Lemma~\ref{lem_good_gamma} holds.
    The result follows from $\pi(x,\gamma_j) \geq \pi(x,\gamma_0) \geq \pi(x,\gamma_{\max})$ (Proposition~\ref{p:pi_properties} with $\gamma_j \leq \gamma_0 \leq \gamma_{\max}$). 
\end{proof}

\begin{remark} \label{rem_practical_backtracking}
In our numerical experiments, follow the approach in Corollary~\ref{cor_ppg_multipleN}, where we accept the current choice of $\gamma_j$ if all computed iterates $u_i$ of Algorithm~\ref{alg_ppg} satisfy $m(0)-m(u_i) > 0$ and $m(0)-m(p^*) > 0$.
For the backtracking, we use the heuristic choice $\gamma_0 = \frac{2\|g_0\|}{3\|H_0 g_0\|}$, where $g_0$ and $H_0$ are the model gradient and Hessian at the first iteration.
This choice arises from setting $\frac{1}{\gamma_0} - L + \frac{\lambda}{2} = 0$ (compare with Lemma~\ref{l:prox descent}) under the conservative assumption $\lambda=-L$, and where $L=\|H_0\|$ is approximated by one iteration of the power method, $\|H_0\| \geq \|H_0 g_0\| / \|g_0\|$.
The successful value $\gamma_j$ in one iteration of Algorithm~\ref{alg_nonsmooth_tr} is then taken to be the value of $\gamma_0$ in the next iteration.
We do not change the value of $N$ across iterations of Algorithm~\ref{alg_nonsmooth_tr}.
\end{remark}

\section{Numerical Experiments} \label{sec_numerics}

In our experiments we compare Algorithm~\ref{alg_nonsmooth_tr} with different choices of subproblem solver.
Our implemented version of Algorithm~\ref{alg_nonsmooth_tr} uses $\Delta_0=1$ and a slightly more sophisticated mechanism for choosing $x_{k+1}$ and $\Delta_{k+1}$, similar to \cite[Algorithm 4.1]{nocedal2006numerical}.
Specifically, we set $x_{k+1}=x_k+p_k$ if $\rho_k \geq 10^{-3}$ and $x_{k+1}=x_k$ otherwise, and 
\begin{align}
    \Delta_{k+1} = \begin{cases} \min\{2\Delta_k, 10^{10}\}, & \text{$\rho_k \geq 0.75$ and $\|p_k\| \geq (1-10^{-5})\Delta_k$,} \\ 0.5\Delta_k, & \rho_k < 0.25, \\ \Delta_k, & \text{otherwise}. \end{cases}
\end{align}
This modification does not affect the complexity theory developed in Section~\ref{sec_complexity}, and was chosen for practicality.
The two choices of subproblem solver we compare for this version of Algorithm~\ref{alg_nonsmooth_tr} are: 
\begin{itemize}
    \item PPG: the projected proximal gradient method developed in Algorithm~\ref{alg_ppg}; and
    \item SPG: the spectral proximal gradient method given in \cite[Algorithm 3]{Baraldi2023a}.
\end{itemize}
The SPG method was chosen for comparison as it was the top-performer among the 5 subproblem solvers compared in \cite{Baraldi2023a} (excluding one method that specifically exploits an $L^1$ regularization structure rather than just using $\prox_{\gamma h}$).
Both methods use comparable problem information: they interact with the model Hessian $H_k$ only through Hessian-vector products, and with the nonsmooth term $h$ through only evaluations of $h$ and its proximity operator.
For PPG we use $\mu_u=2$ and $\alpha=0.9$ ($\alpha$ is used for backtracking; see Corollary~\ref{cor_ppg_multipleN} and Remark~\ref{rem_practical_backtracking}) and for SPG we set $t_{\min}=10^{-12}$, $t_{\max}=10^{12}$, $\overline{\tau}=10^{-5}$, $t_0=1$ and $\tau_k=10^{-3} h_k$ (with all values except $t_0$ taken from \cite{Baraldi2023a}).
We consider the impact of varying the maximum iterations $N$ for both PPG and SPG subproblem solvers, and show results for $N\in\{15, 30, 50\}$, where $N=15$ is the value used in \cite{Baraldi2023a}.

For our test problems \eqref{eq_composite}, we use $h(x)=\|x\|_1$ and take $f$ to be from a collection of 154 unconstrained CUTEst problems \cite{Gould2015,Fowkes2022} with dimension $d\in[2,50]$ based on \cite[Appendix A.1]{Ragonneau2022}.
These problems are listed in Appendix~\ref{app_cutest_problems}.
For all problems we use the exact Hessian $H_k=\grad^2 f(x_k)$ to build our model \eqref{eq_model}.
We run all test problems until we reach first-order optimality level $\pi_k(1) = \pi(x_k,1)\leq 10^{-6}$, capped at $10^4$ iterations of Algorithm~\ref{alg_nonsmooth_tr}.

To compare our (subproblem) solvers, we measure the number of iterations taken to achieve $\pi_k(1) \leq \tau$ for the first time, for some accuracy level $\tau \ll 1$.
We plot both data \cite{More2009} and performance profiles \cite{Dolan2002}.
For solvers $S\in \mathcal{S}$ and problems $P\in \mathcal{P}$, let $K(S,P,\tau)$ be the first iteration with $\pi_k(1) \leq \tau$ (with $K=+\infty$ if this never occurs).
For a given solver $S$ and accuracy level $\tau$, data profiles measure the proportion of problems solved after some fixed number of iterations,
\begin{align}
    d_{S,\tau}(\alpha) = \frac{1}{|\mathcal{P}|} |\{P : K(S,P,\tau) \leq \alpha \}|, \qquad \forall \alpha \geq 0,
\end{align}
and performance profiles measure the proportion of problems solved within some ratio of the fastest solver for that problem,
\begin{align}
    p_{S,\tau}(\alpha) = \frac{1}{|\mathcal{P}|} |\{P : K(S,P,\tau) \leq \alpha \min_{S'\in\mathcal{S}} K(S',P,\tau) \}|, \qquad \forall \alpha \geq 1.
\end{align}
Our results compare PPG and SPG subproblem solvers with $N\in\{15,30,50\}$ for accuracy levels $\tau\in\{10^{-3}, 10^{-6}\}$, with data and performance profiles given in Figure~\ref{fig_main_results}.

\begin{figure}[tb]
  \centering
  \begin{subfigure}[b]{0.48\textwidth}
    \includegraphics[width=\textwidth]{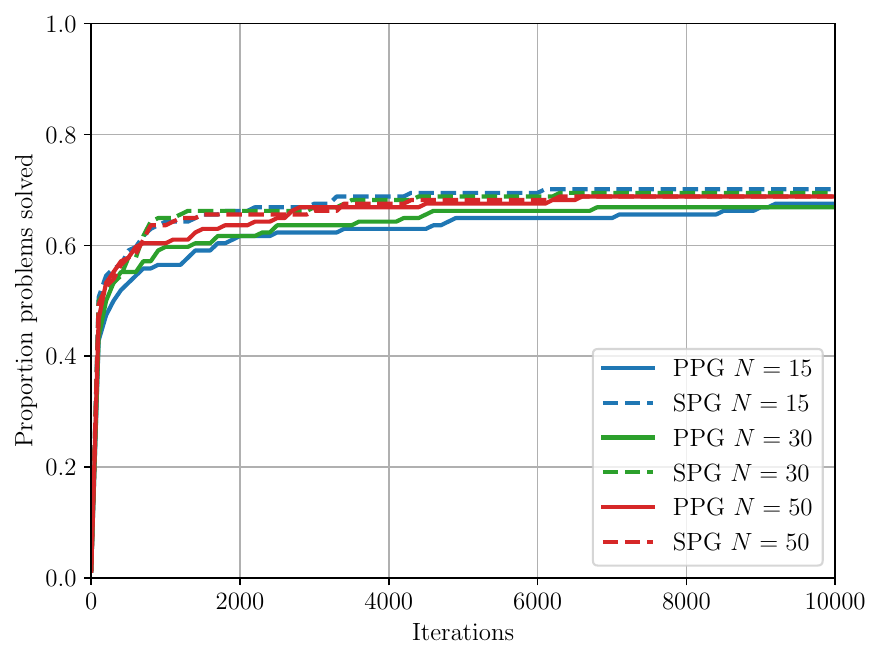}
    \caption{Data profile, $\tau=10^{-3}$}
  \end{subfigure}
  ~
  \begin{subfigure}[b]{0.48\textwidth}
    \includegraphics[width=\textwidth]{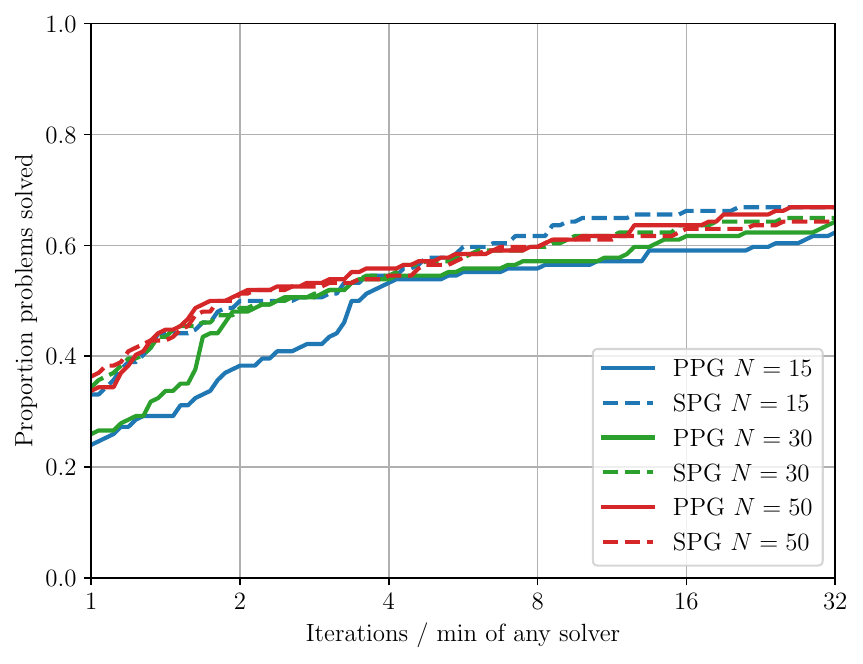}
    \caption{Performance profile, $\tau=10^{-3}$}
  \end{subfigure}
  \\
  \begin{subfigure}[b]{0.48\textwidth}
    \includegraphics[width=\textwidth]{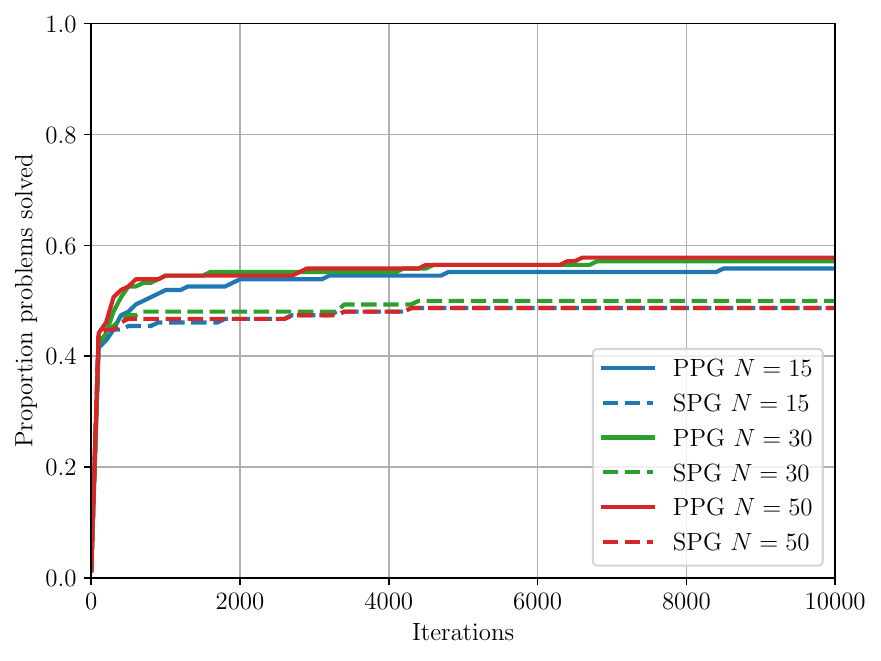}
    \caption{Data profile, $\tau=10^{-6}$}
  \end{subfigure}
  ~
  \begin{subfigure}[b]{0.48\textwidth}
    \includegraphics[width=\textwidth]{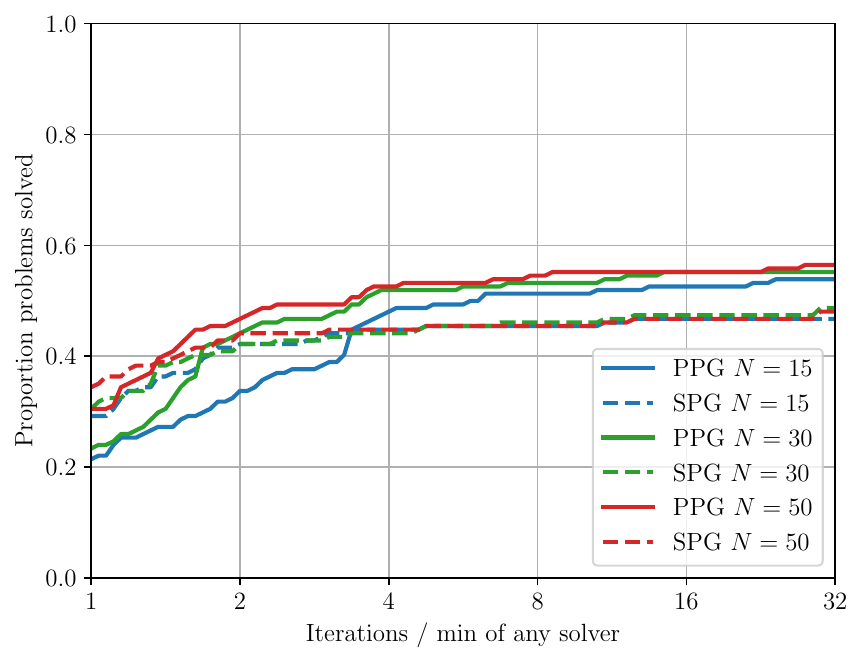}
    \caption{Performance profile, $\tau=10^{-6}$}
  \end{subfigure}
  \caption{Comparison of PPG (Algorithm~\ref{alg_ppg}) and SPG \cite{Baraldi2023a} subproblem solvers with increasing maximum iterations $N\in\{15,30,50\}$ in the trust-region method Algorithm~\ref{alg_nonsmooth_tr}.}
  \label{fig_main_results}
\end{figure}

For low accuracy solutions $\tau=10^{-3}$, the three SPG subproblem solvers outperform the PPG subproblem solvers, with the PPG solvers improving in performance with larger values of $N$.
The best PPG variant, $N=50$, has comparable performance to the SPG solvers.

However, for high accuracy solutions $\tau=10^{-6}$, all the PPG variants notably outperform all the SPG variants.
We also note that increasing the value of $N$ for PPG continues to improve the performance, where increasing $N$ for SPG has limited benefit.

Thus, our results suggest that PPG with larger $N$ is a more effective subproblem solver than PPG with smaller $N$, and that PPG can outperform SPG when higher accuracy solutions are desired.

\section{Conclusion and Future Work}

We extended the theoretical analysis of the nonsmooth trust-region method from \cite{Baraldi2023} to prove worst-case complexity bounds in the case of unbounded model Hessians, and showed results that match the smooth case \cite{Diouane2024}.
We also introduced the projected proximal gradient nonsmooth trust-region subproblem solver, a simple method that demonstrates good numerical performance, particularly when high-accuracy solutions of the main problem \eqref{eq_composite} are desired.
Potential directions for future work include extending our results to the derivative-free case (i.e.~where $\grad f$) is not available, and handling more complicated $h$ via inexact proximity operator evaluations.

\subsection*{Acknowledgments}
This research was initiated during HMP's visit to the Sydney Mathematical Research Institute (SMRI), Australia in May 2024 and HMP and MND's visit to the Vietnam Institute for Advanced Study in Mathematics (VIASM), Vietnam in June 2024, for which HMP and MND are grateful for the generous support and hospitality of SMRI and VIASM. MND is partially supported by the Australian Research Council Discovery Project DP230101749. 
LR is supported by the Australian Research Council Discovery Early Career Award DE240100006.

\bibliographystyle{alpha}
\bibliography{refs}

\appendix

\clearpage

\section{List of Test Problems} \label{app_cutest_problems}
The following table contains a list of the 154 CUTEst problems \cite{Gould2015,Fowkes2022} used for the numerical experiments in Section~\ref{sec_numerics} and their dimension $d\in[2,50]$, based on the collection from \cite[Appendix A.1]{Ragonneau2022}.
Any values in brackets after the problem name are the  optional problem parameters used.

\begin{table}[H]
    \centering
    {\scriptsize
    \begin{tabular}{cc|cc|cc}
        Name (params) & $d$ & Name (params) & $d$ & Name (params) & $d$ \\ \hline
AKIVA & 2 & FREUROTH ($N=10$) & 10 & PALMER2C & 8 \\
ALLINITU & 4 & GAUSSIAN & 3 & PALMER3C & 8 \\
ARGLINA ($N=10$) & 10 & GBRAINLS & 2 & PALMER4C & 8 \\
ARGLINB ($N=10$) & 10 & GENHUMPS ($N=5$) & 5 & PALMER5C & 6 \\
ARGLINC ($N=10$) & 10 & GENROSE ($N=5$) & 5 & PALMER5D & 4 \\
ARGTRIGLS ($N=10$) & 10 & GROWTHLS & 3 & PALMER6C & 8 \\
BARD & 3 & GULF & 3 & PALMER7C & 8 \\
BEALE & 2 & HAHN1LS & 7 & PALMER8C & 8 \\
BENNETT5LS & 3 & HAIRY & 2 & PENALTY1 ($N=10$) & 10 \\
BIGGS6 & 6 & HATFLDD & 3 & PENALTY2 ($N=10$) & 10 \\
BOX ($N=10$) & 10 & HATFLDE & 3 & POWELLBSLS & 2 \\
BOX3 & 3 & HATFLDFL & 3 & POWELLSG ($N=16$) & 16 \\
BOXBODLS & 2 & HEART6LS & 6 & POWER ($N=20$) & 20 \\
BOXPOWER ($N=10$) & 10 & HEART8LS & 8 & QUARTC ($N=25$) & 25 \\
BROWNAL ($N=10$) & 10 & HELIX & 3 & RAT42LS & 3 \\
BROWNBS & 2 & HIELOW & 3 & ROSENBR & 2 \\
BROWNDEN & 4 & HILBERTA ($N=5$) & 5 & ROSENBRTU & 2 \\
BROYDN3DLS ($N=10$) & 10 & HILBERTB ($N=10$) & 10 & ROSZMAN1LS & 4 \\
BROYDNBDLS ($N=10$) & 10 & HIMMELBB & 2 & S308 & 2 \\
BRYBND ($N=10$) & 10 & HIMMELBF & 4 & SBRYBND ($N=10$) & 10 \\
CHNROSNB ($N=25$) & 25 & HIMMELBG & 2 & SCHMVETT ($N=3$) & 3 \\
CHNRSNBM ($N=25$) & 25 & HIMMELBH & 2 & SCOSINE ($N=10$) & 10 \\
CHWIRUT1LS & 3 & HUMPS & 2 & SCURLY10 ($N=10$) & 10 \\
CHWIRUT2LS & 3 & INDEFM ($N=10$) & 10 & SENSORS ($N=3$) & 3 \\
CLIFF & 2 & JENSMP & 2 & SINEVAL & 2 \\
COSINE ($N=10$) & 10 & KIRBY2LS & 5 & SINQUAD ($N=5$) & 5 \\
CUBE & 2 & KOWOSB & 4 & SISSER & 2 \\
DENSCHNA & 2 & LANCZOS1LS & 6 & SNAIL & 2 \\
DENSCHNB & 2 & LANCZOS2LS & 6 & SPARSINE ($N=10$) & 10 \\
DENSCHNC & 2 & LANCZOS3LS & 6 & SPARSQUR ($N=10$) & 10 \\
DENSCHND & 3 & LIARWHD ($N=36$) & 36 & SSBRYBND ($N=10$) & 10 \\
DENSCHNE & 3 & LOGHAIRY & 2 & SSCOSINE ($N=10$) & 10 \\
DENSCHNF & 2 & LSC1LS & 3 & SSI & 3 \\
DIXON3DQ ($N=10$) & 10 & LSC2LS & 3 & STREG & 4 \\
DJTL & 2 & MANCINO ($N=20$) & 20 & THURBERLS & 7 \\
DQDRTIC ($N=10$) & 10 & MARATOSB & 2 & TOINTGOR & 50 \\
DQRTIC ($N=10$) & 10 & MEXHAT & 2 & TOINTGSS ($N=10$) & 10 \\
ECKERLE4LS & 3 & MEYER3 & 3 & TOINTPSP & 50 \\
EDENSCH ($N=36$) & 36 & MGH09LS & 4 & TOINTQOR & 50 \\
ENGVAL1 ($N=2$) & 2 & MGH10LS & 3 & TQUARTIC ($N=10$) & 10 \\
ENGVAL2 & 3 & MISRA1BLS & 2 & TRIDIA ($N=20$) & 20 \\
ENSOLS & 9 & MISRA1DLS & 2 & VARDIM ($N=10$) & 10 \\
ERRINROS ($N=25$) & 25 & MOREBV ($N=10$) & 10 & VAREIGVL ($N=19$) & 20 \\
ERRINRSM ($N=25$) & 25 & NCB20B ($N=22$) & 22 & VESUVIALS & 8 \\
EXPFIT & 2 & NONCVXU2 ($N=10$) & 10 & VESUVIOLS & 8 \\
EXTROSNB ($N=5$) & 5 & NONCVXUN ($N=10$) & 10 & VESUVIOULS & 8 \\
FBRAIN3LS & 6 & NONDIA ($N=20$) & 20 & VIBRBEAM & 8 \\
FLETBV3M ($N=10$) & 10 & OSBORNEB & 11 & WATSON ($N=12$) & 12 \\
FLETCBV2 ($N=10$) & 10 & OSCIGRAD ($N=10$) & 10 & YFITU & 3 \\
FLETCBV3 ($N=10$) & 10 & OSCIPATH ($N=5$) & 5 & ZANGWIL2 & 2 \\
FLETCHBV ($N=10$) & 10 & PALMER1C & 8 & & \\
FLETCHCR ($N=10$) & 10 & PALMER1D & 7 & & \\
    \end{tabular}
    }  % end font size
    \caption{List of CUTEst problems used for numerical experiments.}
\end{table}

\end{document}